\documentclass[11pt,a4paper,reqno]{amsart}

\usepackage[utf8]{inputenc}
\usepackage[T1]{fontenc}
\usepackage{amsmath,amssymb,amsthm}
\usepackage{mathtools}
\usepackage{graphicx}
\usepackage{hyperref}
\usepackage{xcolor}
\usepackage[margin=1in]{geometry}
\usepackage{mathrsfs}
\usepackage{enumitem}

\newtheorem{theorem}{Theorem}[section]
\newtheorem{lemma}[theorem]{Lemma}
\newtheorem{proposition}[theorem]{Proposition}
\newtheorem{corollary}[theorem]{Corollary}

\theoremstyle{definition}
\newtheorem{definition}[theorem]{Definition}

\newtheorem{remark}[theorem]{Remark}

\theoremstyle{plain}

\numberwithin{equation}{section}

\newcommand{\R}{\mathbb{R}}
\newcommand{\C}{\mathbb{C}}
\newcommand{\Z}{\mathbb{Z}}
\newcommand{\N}{\mathbb{N}}

\newcommand{\norm}[1]{\left\lVert #1 \right\rVert}

\title[Control for the Schr\"{o}dinger equation on polyhedra]{Observability and controllability for the Schr\"{o}dinger equation on polyhedra}

\author[Wei Qu]{Wei Qu}
\address{School of Mathematics and Statistics, Huazhong University of Science and Technology, 430074, Wuhan, China}
\email{quwei@hust.edu.cn}

\author[Zhiwen Duan]{Zhiwen Duan}
\address{School of Mathematics and Statistics, Huazhong University of Science and Technology, 430074, Wuhan, China}
\email{duanzhw@hust.edu.cn}

\begin{document}
	
	\begin{abstract}
		We consider the observability and controllability for the Schr\"{o}dinger equation on polyhedra. We show that for the Schr\"{o}dinger equation on convex polyhedra in $\mathbb{R}^d\,(d\geq2)$, when the initial data possess $H^s\,(s>d/2)$ regularity, observability and controllability hold, with the control region being an arbitrary nonempty open neighborhood of the singular set of the boundary, which is the same as the control region for eigenfunctions. Our proof of observability is a concrete adaptation of the method in the Burq--Zworski black box control work, where the observability inequality is proved via a resolvent estimate. Hence, before proving observability, we present an observability resolvent estimate on polyhedra. The idea of the proof of this resolvent estimate comes from the case of eigenfunction concentration on polyhedra considered by Ceki\'c--Georgiev--Mukherjee, which uses semiclassical measures and relies closely on the dynamical properties of the billiard flow in polyhedra. Finally, by a standard HUM argument, we show that observability implies controllability.
	\end{abstract}
	
	\maketitle
	
	\section{Introduction}
	\label{sec: Introduction}
	In this paper we consider the observability and controllability for the Shr\"odinger equation on polyhedra in $\R^d\,(d\geq2):$
	\begin{equation}
		\left\{
		\begin{aligned}
			i\partial_tu(t,z)
			&+\Delta u(t,z)=0,\quad (t,z)\in\R_t\times P,\\
			u(0,z)
			&=u_0\in L^2(P),\quad u|_{\R_t\times\partial P}=0,
		\end{aligned}
		\right.
	\end{equation}
	where $P\subset\R^d$ is a convex polyhedron. We say that the quantum system described by the Schrödinger equation on a polyhedron $P$ is \emph{observable} if solutions of the Schrödinger equation satisfy the following observability inequality:
	\begin{equation}
		\norm{u(0,\cdot)}_{L^2(P)}^{2}\leq C\int_{0}^{T}\norm{u(t,\cdot)}_{L^2(\omega)}^2dt,
	\end{equation}
	where $\omega\subset P$ is a measurable set referred to as the \emph{control region}, and $C=C(\omega,T)>0$ is a constant depending on $\omega$ and $T.$ Intuitively, the observability inequality means that the quantum particle described by the Schrödinger equation must leave a trace on the control region $\omega$ over the time interval $[0,T]$. From the perspective of control theory, the observability inequality implies that when we know the observation data of the system on the region $\omega$ over the time interval $[0,T],$ we can uniquely determine the state of the system at the initial time. In control theory, the observability is often closely related to the controllability. A system is said to be \emph{controllable} if, for any two given state functions $u_0$ and $u_1$ in the state space, we can find a control term applied to the control region such that when the system evolves with $u_0$ as the initial state, it exactly reaches the state $u_1$ at some time $T.$ Expressed in terms of equations, this is:
	For any $u_0,u_1\in L^2(P),$ does there exist a control function $g\in L^2([0,T]\times P)$ such that the soution $u$ of the following equation
	\begin{equation}
		\left\{
		\begin{aligned}
			i\partial_{t}u(t,z)&+\Delta u(t,z)=\mathbf{1}_{[0,T]\times\omega}g(t,z),\\
			u(0,z)&=u_0,\quad u|_{\R_t\times\partial P}=0
		\end{aligned}
		\right.
	\end{equation}
	satisfies $u(T,z)\equiv u_1$? By the linearity and time reversibility of the Schr\"{o}dinger equation, controllability is equivalent to null controllability, which means $u_1\equiv0$ above (see for example \cite{Laur14}). By a classical Hilbert Uniqueness Method (HUM) argument we can see that observability and null controllability are equivalent.
	
	There is extensive literature on the observability and control for the Schrödinger equation. Pioneering work on the observability and control for the Schrödinger equation was given by Lebeau in \cite{Lebeau92} under the following \emph{Geometric Control Condition (GCC)}:
	\begin{quote}
		\centering
		There exists $L=L(\omega)>0$ such that\\
		every geodesic of length $L$ intersects the control region $\omega.$
	\end{quote}
	
	The GCC originally proposed in \cite{BLR92} was proved to be sufficient but not necessary for the observability of the Schr\"odinger equation. A classic case in which the Schrödinger equation is observable while the control region does not satisfy the geometric control condition is the flat torus. In general, any nonempty open set on flat torus can serve as a control region for which the observability inequality holds; see Jaffard~\cite{Jaffard} and Haraux~\cite{Haraux} for the two-dimensional case and Komornik~\cite{Komornik} for higher dimensions. Burq--Zworski~\cite{BZ12} proved this result for Schr\"{o}dinger operators with smooth potentials on two-dimensional torus, and later Bourgain--Burq--Zworski~\cite{BBZ13} proved the case with $L^2$-potentials. Anantharaman--Maci\`a~\cite{AM14} extended this result to the case of higher-dimensional tori with some class of potentials including continuous ones. Also, Burq--Zworski~\cite{BZ19} presented another result showing that on two-dimensional torus any set of positive Lebesgue measure can be used for observability. Another important example where the control region does not satisfy the geometric control condition is the unit disk in $\R^2.$ In this case, Anantharaman--Léautaud--Macià~\cite{ALM16} proved that the observability inequality holds if and only if the control region is an arbitrary nonempty open set touching the boundary.
	
	However, the geometric control condition has been proved to be necessary for observability on Zoll manifolds (i.e., manifolds all of whose geodesics are closed) due to the strong stability of geodesic flow on such manifolds (see \cite{Maci09}\cite{Maci11}). In this case, if we choose a control region that does not satisfy the geometric control condition, then there exists a geodesic that does not pass through this region, together with a sequence of Gaussian beam solutions whose $L^2$-mass concentrates around that geodesic in the high-frequency limit, which will destroy the observability.
	
	For the case of negatively curved manifolds, Anantharaman--Rivi\`ere~\cite{AR12} gave a result on the controllability of the Schr\"{o}dinger equation on a compact Riemannian manifold of constant negative curvature under an entropy condition, namely, that the set of uncontrolled trajectories is ``thin''. Dyatlov--Jin~\cite{DJ18} proved that on a compact Riemannian surface of constant negative curvature, any nonempty open set achieves global controllability in the $L^2$-sense, while Jin~\cite{J18} used this result to prove an control result for the time-dependent Schr\"{o}dinger equation on the same manifold, with the same control region as in \cite{DJ18}. Later, Dyatlov--Jin--Nonnenmacher~\cite{DJN} established the same result on compact Riemannian surface of variable negative curvature.
	
	In the two classes of examples mentioned above where the observability inequality holds but the control region fails to satisfy the GCC, the geodesic flows are all completely integrable in phase space. There exists another class of examples in which the observability inequality holds and the control region does not satisfy the GCC, but whose geodesic flow is not completely integrable, namely, partially rectangular domains (including the Bunimovich stadium). For domains of this type, the control region is an arbitrary open neighborhood of the lateral wings. Due to the existence of billiard trajectories bouncing between the two parallel sides of the rectangular part, the control region does not satisfy the geometric control condition (see\cite{BZ04}\cite{BZ05}\cite{Mar06}). Using this result, Hassell--Hillairet--Marzuola~\cite{HHM09} gave a result on the concentration of eigenfunctions of the Laplacian on polygonal domains, showing that the $L^2$-mass of eigenfunctions cannot concentrate away from any neighborhood of the set of vertices. Ceki\'c--Georgiev--Mukherjee~\cite{CGM} extended this result to higher-dimensional convex polyhedra, where the control region is an arbitrary nonempty neighborhood of the singular set of the boundary. In both of the above cases, it is necessary to analyze the dynamical properties of the billiard flow inside the domain.
	
	In this paper, we present results on the observability and controllability for the time-dependent Schr\"{o}dinger equation on convex polyhedra in $\R^d\,(d\geq2),$ with the same control region as in \cite{CGM}. For a convex polyhedron $P$ in $\R^d$, let $\mathcal{S}$ be the singular set of the boundary $\partial P$ (i.e., the non‑smooth part of the boundary; see Section~\ref{sec: Billiard Dynamics on Polyhedra}). Then we have the following observability result:
	\begin{theorem}
		\label{thm1.1}
		Suppose that $U$ is a nonempty open neighborhood of $\mathcal{S}$ inside $P.$ Then there exists $T_0>0$ such that for any $T>T_0$ and any solution of
		\begin{equation}
			\left\{
			\begin{aligned}
				i\partial_tu(t,z)
				&+\Delta u(t,z)=0,\quad (t,z)\in\R_t\times P,\\
				u(0,z)
				&=u_0,\quad u|_{\R_t\times\partial P}=0,
			\end{aligned}
			\right.
		\end{equation}
		where $u_0\in H^s(P)\,(s>d/2),$ there exists a constant $C=C(U,T)>0$ such that
		\begin{equation}
			\label{eq: obsr in P}
			\norm{u_0}^2_{L^2(P)}\leq C\int_{0}^{T}\norm{u(t,z)}^2_{L^2(U)}dt.
		\end{equation}
	\end{theorem}
	\noindent By a classical HUM argument, we have the following controllability result:
	\begin{theorem}
		\label{thm1.2}
		Suppose that $U$ is a nonempty open neighborhood of $\mathcal{S}$ inside $P.$ Then there exists $T_0>0$ such that for any $T>T_0,$ there exists $g\in L^2([0,T];H^s(P))\,(s>d/2)$ such that the solution of
		\begin{equation}
			\left\{
			\begin{aligned}
				i\partial_{t}u(t,z)&+\Delta u(t,z)=\mathbf{1}_{[0,T]\times U}g(t,z),\quad(t,z)\in\R_t\times P,\\
				u(0,z)&=u_0,\quad u|_{\R_t\times\partial P}=0,
			\end{aligned}
			\right.
		\end{equation}
		where $u_0\in H^s(P),$ satisfies
		\[u(T,\cdot)\equiv0.\]
	\end{theorem}
	Here we require that $u_0\in H^s(P)\,(s>d/2)$ so that, by the Sobolev embedding theorem, $u_0$ and hence $u$ are continuous in the spatial variables. The reason why we require the solution $u$ to be continuous in the spatial variables is explained in Remark~\ref{remark: why s>d/2}. This then raises the question of whether the same conclusion holds when $u_0$ merely belongs to $L^2(P)$.
	
	The proof of our main result also requires the dynamical properties of the billiard flow in convex polyhedra, which are known results and are presented in Section~\ref{sec: Billiard Dynamics on Polyhedra}. The content of Section~\ref{sec: Billiard Dynamics on Polyhedra} shows that any billiard trajectory inside a convex polyhedron that does not meet the singular set $\mathcal{S}$ is contained in some ``immersed periodic tube'' with a suitable cross-section, and there are only finitely many such periodic tubes. The billiard trajectories contained in these periodic tubes are not necessarily individually periodic, but the Poincar\'e map on the cross-section of a periodic tube, after a certain iteration, is a rotation. Consequently, when we consider functions within a periodic tube, they may not be periodic along the direction of the tube but are almost periodic; thus, in Section~\ref{sec: Almost Periodic Functions and Almost Periodic Boundary Conditions} we provide relevant material on almost periodic functions and almost periodic boundary conditions. In Section~\ref{sec: An Observability Resolvent Estimate} we give an observability resolvent estimate for the stationary Schr\"odinger equation with a source term in polyhedra, which is crucial for the proof of Theorem~\ref{thm1.1}. The idea of the proof of this resolvent estimate comes from \cite{CGM}; it uses semiclassical measures and is divided into the rational polyhedra case and the general case, based on the dynamical properties of the billiard flow. In Section~\ref{sec: Proof of Theorem1.1} we present the proof of Theorem~\ref{thm1.1}, which is an appropriate adaptation of the method of Burq--Zworski in \cite{BZ04}. Finally, in Section~\ref{sec: From Observability to Control} we give the proof of Theorem~\ref{thm1.2} by a standard HUM argument.
	
	\section{Preliminaries}
	\subsection{Billiard Dynamics in Polyhedra}
	\label{sec: Billiard Dynamics on Polyhedra}
	In this subsection we present some known facts about the dynamical properties of the billiard flow in polyhedra in $\R^d\,(d\geq2),$ which will be used later in the proofs. All these results and their proofs can be found in \cite{CGM}; for the sake of completeness and readability, we quote them here without proof.
	
	Let $P\subset\R^d$ be a convex polyhedron. Denote by $\mathcal{S}$ the \emph{singular set} of the boundary $\partial P;$ that is, $\mathcal{S}\subset\partial P$ is the union of faces of $\partial P$ with dimension $\leq d-2$, or in other words, the $(d-2)$-skeleton of $\partial P.$ For example, when $d=2$, the singular set of the boundary of a two-dimensional polyhedron (i.e., a polygon) is the set of vertices of the polygon; when $d=3$, the singular set of the boundary of a three-dimensional polyhedron is the union of the edges of the polyhedron.
	
	The billiard flow in $P$ describes the motion of a point particle that moves in straight lines with unit speed inside $P.$ When the billiard hits one of the faces of the boundary, it obeys the law of optical reflection (i.e., the angle of incidence equals the angle of reflection). Trajectories which hit the singular set $\mathcal{S}$ stop right there; such trajectories are also called singular. Such trajectories form a set of measure zero. If a trajectory never hits the singular set, then it can be extended for all time. Now we introduce some notations and definitions about this system, which is essentially the same as in \cite{CGM}.
	
	Let $\Gamma:=\partial P$, and denote by $T\Gamma$ the set of all unit tangent vectors with base points in $\Gamma$ and pointing inside $P$. Let $T\Gamma_1:=\{x\in T\Gamma : \text{the forward orbit of } x \text{ never hits } \mathcal{S}\}.$ Denote by $f$ the first return (Poincar\'e) map of the billiard flow to the set $T\Gamma.$ Then $f$ and its iterates are defined smoothly on $T\Gamma_1$. Suppose the polyhedron $P$ has $l$ faces denoted by $\mathscr{F}_1,\mathscr{F}_2,\dots,\mathscr{F}_l.$ Let $\Sigma^{+}_{l}:=\{1,2,\dots,l\}^{\N}$ represent the set of all \emph{forward strings} for the symbolic dynamics of the billiard flow.
	
	For a billiard trajectory determined by $x\in T\Gamma_1,$ the \emph{symbolic string} for the forward orbit is given by $w(x),$ defined by $w(x)_i=j$ iff the basepoint of $f^i(x)$ lies in $\mathscr{F}_j$ (here the subscript $i$ denotes the $i$-th component of $w(x)$, and the superscript $i$ denotes the $i$-th iteration of $f$). The set of all these possible symbolic strings is denoted by $\Sigma^{+}_{P}:=\{w\in\Sigma^{+}_{l}:\exists x\in T\Gamma_1\text{ such that } w=w(x)\}.$ For a given symbolic string $w\in\Sigma^{+}_{P},$ let $X(w):=\{x\in T\Gamma_1:w(x)=w\}$ denote all tangent vectors whose billiard trajectories have the same symbolic representation $w.$ An arbitrary element of $X(w)$ is denoted by $x(w).$
	
	When $d=2,$ polyhedra (i.e., polygons) can be divided into rational and irrational ones, depending on whether the vertex angles are rational. A two-dimensional polyhedron is called \emph{rational} if all its vertex angles are rational; otherwise, it is called \emph{irrational}. When $d\geq3,$ there is no standard definition of rational or irrational polyhedra. We present the definition in \cite{CGM}, which is formulated from the viewpoint of dynamical systems. There is an important difference between billiard trajectories in rational polyhedra and those in irrational polyhedra, which we will present later.
	\begin{definition}
		Let $P\subset\R^d$ be a polyhedron and let $\rho_i$ represent the linear reflection determined by the $i^{\text{th}}$-face of $P.$ Then $P$ is called \emph{rational} if the group $G$ generated by $\rho_i$ is finite, otherwise the polyhedron is called \emph{irrational}.
	\end{definition}
	
	There is an unfolding process for billiard trajectories in a polyhedron. Let $\gamma$ be a billiard trajectory in $P$. When $\gamma$ hits a boundary face $\mathscr{F}_i$ of $P$, we reflect $P$ across $\mathscr{F}_i$ and denote the reflected polyhedron by $\sigma P$ where $\sigma$ denotes the reflection. We identify the corresponding points on the common face of $P$ and $\sigma P$ under the reflection, so that the trajectory $\gamma$ continues into $\sigma P$ along the original direction until it hits a next face, and the process is repeated. In this way, the trajectory $\gamma$ can be viewed as a straight line without reflection in the doubled polyhedron. We call this process the \emph{unfolding} of the trajectory.
	
	Let $D:=(P\cup\sigma P)/\sim$ denote the double of $P$, where $\sim$ denotes the gluing by a pointwise identification of the common face $\partial P\cap\partial(\sigma P).$ Let $D_0:=D\setminus\mathcal{S}.$ Given a set $U\subset\R^{d-1}$ and a local isometry $F:U\times\R\rightarrow D_0,$ we call $F(U\times\R)$ an \emph{immersed tube} or just a \emph{tube}. We will often identify $F$ with its image $\mathscr{C}:=F(U\times\R).$ We call $U$ the \emph{cross-section} of the tube $\mathscr{C}.$ A \emph{lifted tube} is the image of $\mathscr{C}$ in the unit sphere bundle $SD_0,$ determined by the unit vector in the positive direction of the tube. An immersed tube can be specified by a subset $Q$ of $T\Gamma_1,$ consisting of parallel vectors whose base points form a convex set on one of the faces of the polyhedron. For a trajectory $\gamma$ generated by some $x\in T\Gamma_1$ with symbolic representation $w(x)\in\Sigma^{+}_{P},$ if $\gamma$ does not hit the singular set $\mathcal{S},$ then it can be ``thickened'' to form a tubular neighborhood (i.e., an immersed tube) $\mathscr{C}$ around $\gamma$ such that each trajectory in $\mathscr{C}$ parallel to $\gamma$ has the same symbolic representation $w(x).$ This time the intersection of the lifted tube of $\mathscr{C}$ with $T\Gamma_1$ is $Q=X(w),$ and we call such a tube $\mathscr{C}$ \emph{maximal} since it cannot be enlarged. Such maximal tubes may have some kinds of periodicity depending on the polyhedron, which we will present later. Now we give the definition of periodic tubes.
	\begin{definition}
		\label{def: periodic tube}
		Let $U\subset\R^{d-1}$ and let $F:U\times\R\rightarrow D_0$ be a local isometry. Then an immersed tube $F(U\times\R)$ is called \emph{periodic} if there is a positive number $L$ and a rotation $\mathscr{R}$ in $\R^d$ fixing the $\R$ direction of $U\times\R$, such that $F(x,y+L)=F(\mathscr{R}x,y)$ for all $(x,y)\in U\times\R.$
	\end{definition}
	
	In the above definition, the $L$ is referred to as the length of the periodic tube, which is also a period of the \emph{central geodesic} given by $F(\{x_0\}\times\R),$ where $x_0\in U$ is the center of mass of $U$; the $\mathscr{R}$ is called the \emph{rotation associated with the periodic tube}. It should be pointed out that the periodic tube in the above definition does not mean that all the parallel billiard trajectories contained in the tube are \emph{individually} periodic.
	
	Periodic tube can also be ``thickened'' to be maximal. For a maximal periodic tube in $D_0$ we have the next properties:
	\begin{proposition}{\cite[Proposition~3.5]{CGM}}
		\label{thm: boundary arbitrarily close to singular set}
		The cross-section of a maximal periodic tube $\mathscr{C}$ is convex. Every trajectory on the boundary of $\mathscr{C}$ comes arbitrarily close to the singular set.
	\end{proposition}
	
	The rationality or irrationality of $P$ affects the shape of the cross-section of a maximal periodic tube in $D_0$ and the periodicity of the periodic tube. For example, when $d=3$, for a rational polyhedron, the cross-section of a maximal periodic tube in $D_0$ is a convex polygon, the associated rotation $\mathscr{R}$ is a rotation by a rational angle, and the parallel billiard trajectories contained in the periodic tube are all individually periodic; whereas for an irrational polyhedron, the cross-section of a maximal periodic tube may be a disk; in this case the associated rotation $\mathscr{R}$ is a rotation by an irrational angle, and among the parallel billiard trajectories contained in the periodic tube only the central geodesic is periodic (see \cite{CGM}\cite{GKT}).
	
	As mentioned before, if a billiard trajectory $\gamma$ in $D_0$ generated by $x\in T\Gamma_1$ with a symbol $w$ does not hit $\mathcal{S},$ it can be ``thickened'' to form a tube $\mathscr{C}$ around $\gamma$ such that all trajectories contained in $\mathscr{C}$ have the same symbol $w.$ If $w$ is periodic, such a tube $\mathscr{C}$ has some periodicity described by the following proposition:
	\begin{proposition}{\cite[Theorem~3.7]{CGM}}
		\label{thm: w periodic}
		Let $P\subset\R^d$ be an arbitrary convex polyhedron and $w\in\Sigma^{+}_{P}$ is a periodic sequence with minimal period $k.$ Then the following hold:
		\begin{enumerate}[label=(\arabic*), nosep]
			\item There exists $x(w)$ such that $x(w)$ is periodic with minimal period $k.$
			\item In addition, one of the following two cases holds:
			\begin{enumerate}[nosep]
				\item There exists $q\geq1$ such that all $y(w)\in X(w)\setminus x(w)$ are periodic with period $qk$ and the cross-section of the tube generated by $X(w)$ is an open polyhedron.
				\item The set $X(w)$ generates a periodic tube with a convex cross-section $\Omega\subset\R^{d-1}$ and an associated isometry $\mathscr{R}_0\in O(d-1)$ keeping $\Omega$ invariant.
			\end{enumerate}
			\item If $d=2$ or $3$ and $k$ is odd, then only the case (2).(a) above can happen and $q=2.$
			\item If $P$ is rational then only the case (2).(a) above can happen.
		\end{enumerate}
	\end{proposition}
	\begin{remark}
		Proposition~\ref{thm: w periodic} is a higher-dimensional generalization of \cite[Theorem~5]{GKT}. The isometry $\mathscr{R}_0$ in part (2).(b) of the proposition belongs to $O(d-1)$ and is not always a rotation. However, at the cost of doubling the length of the periodic tube, we may always assume that $\mathscr{R}_0$ is a rotation.
	\end{remark}
	
	It can also be seen from the above proposition that even if the symbol $w$ is periodic, the parallel trajectories contained in the periodic tube generated by $X(w)$ are not necessarily individually periodic.
	
	If $w$ is non-periodic, we have the following result:
	\begin{proposition}{\cite[Theorem~3.9]{CGM}}
		\label{thm: w non-periodic}
		Let $P\subset\R^d$ be a convex polyhedron and $w\in\Sigma^{+}_{P}$ non-periodic. Let $x\in X(w).$ Then the closure of the trajectory generated by $x$ intersects the singular set $\mathcal{S}.$
	\end{proposition}
	
	Combining Proposition~\ref{thm: w periodic} and Proposition~\ref{thm: w non-periodic} we can get the following dichotomy:
	\begin{corollary}{\cite[Corollary~3.10]{CGM}}
		\label{thm: dichotomy}
		For any billiard trajectory $\gamma$ in a convex polyhedron $P,$ either $\gamma$ is contained in an immersed periodic tube, or the closure of $\gamma$ intersects $\mathcal{S}.$
	\end{corollary}
	When $d=2,$ in \cite{HHM09} the authors proved that any sufficiently small neighborhood $U_{\varepsilon}$ of the singular set $\mathcal{S}$ (i.e., the set of vertices) of a two-dimensional polyhedron (i.e., a polygon) satisfies the following two properties (which they called the ``cylinder condition''):
	\begin{enumerate}[label=(P\arabic*), nosep]
		\item For any geodesic $\gamma,$ either $\gamma$ is periodic, or the closure of $\gamma$ intersects $\mathcal{S}.$
		\item Any geodesic that avoids $U_{\varepsilon}$ belongs to a maximal cylinder and the number of such maximal cylinders is finite.
	\end{enumerate}
	In \cite{CGM}, the authors generalized the above two properties to higher dimensions. The generalization of (P1) to higher-dimensional polyhedra can be obtained from Corollary~\ref{thm: dichotomy} above, namely any billiard trajectory in a polyhedron that does not intersect a small neighborhood of the singular set $\mathcal{S}$ is necessarily contained in some periodic tube. There is also a generalization of (P2) to higher-dimensional polyhedra, which is described through the following definition:
	\begin{definition}
		\label{def: finite tube condition}
		Let $D$ be the double of $P$ as defined earlier. A region $U\subset D$ is said to satisfy the \emph{finite tube condition} if there exists a finite collection of periodic tubes $(\mathscr{C}_i)_{1\leq i\leq n}$ such that any trajectory that avoids $U$ belongs to some $\mathscr{C}_i.$
	\end{definition}
	By Definition~\ref{def: finite tube condition} we have the following proposition which is the generalization of Property~(P2) to higher-dimensional polyhedra:
	\begin{proposition}{\cite[Theorem~3.12]{CGM}}
		\label{thm: finite tube condition}
		Let $P\subset\R^d$ be a convex polyhedron and let $D$ be its double. Then for any $\varepsilon>0,$ the $\varepsilon$-neighborhood $U_{\varepsilon}$ of the singular set $\mathcal{S}$ of $D$ satisfies the finite tube condition.
	\end{proposition}
	
	We will use the weaker Property~(P2') below, which was introduced in \cite{CGM}:
	\begin{quote}
		\centering
		\emph{Property~(P2').} There exists $\delta>0,$ $l>0,$ and $\eta>0,$ such that for each $(x_0,y_0)\in(\Omega_{\varepsilon/2-\eta}\setminus\Omega_{\varepsilon/2+\eta})\times\R,$ there exists a point $(x_0,y)$ with $|y-y_0|\leq l,$ and the $\delta$-neighborhood of $(x_0,y)$ is contained in $F^{-1}(U_{\varepsilon}).$
	\end{quote}
	Property~(P2') says that in each maximal periodic tube $\mathscr{C}_i$ in $D_0,$ for any small neighborhood $U_{\varepsilon}$ of the singular set $\mathcal{S}$, there exists a conic neighborhood in the cotangent bundle whose base set is the annular region $(\Omega_{\varepsilon/2-\eta}\setminus\Omega_{\varepsilon/2+\eta})\times\mathbb{R}$ contained in some neighborhood of $\partial\mathscr{C}_i$, with directions lying near the direction of the tube, such that every point in this conic neighborhood hits $U_{\varepsilon}$ in finite time under the billiard flow. For Property~(P2') we have the following proposition:
	\begin{proposition}{\cite[Lemma~3.14]{CGM}}
		Property~(P2') holds for any $d\geq2.$
	\end{proposition}
	
	\subsection{Almost Periodic Functions and Almost Periodic Boundary Conditions}
	\label{sec: Almost Periodic Functions and Almost Periodic Boundary Conditions}
	In this subsection, we present some basic concepts of almost periodic functions and related harmonic analysis. Similar to periodic functions, almost periodic functions also have a harmonic analysis theory. It was mainly founded by Bohr, and later generalized and refined by Bochner, von~Neumann, Besicovitch, and others. The core idea is to extend the Fourier series theory of periodic functions to those functions that are almost periodic. For more details, see \cite{LZ82}.
	
	Let $X$ be a Banach space with norm $\norm{\cdot}.$ We say a number $\tau\in\R$ is an \emph{$\varepsilon$-almost period} of $f:\R_y\rightarrow X$ if
	\begin{equation}
		\label{eq: varepsilon-almost period}
		\sup_{y\in\R_y}\norm{f(y+\tau)-f(y)}\leq\varepsilon.
	\end{equation}
	We now give the definition of almost periodic functions:
	\begin{definition}
		\label{def: almost periodic function}
		A continuous function $f:\R_y\rightarrow X$ is called \emph{almost periodic} if for any $\varepsilon>0,$ there is an $l=l(\varepsilon)>0$ such that for each $\alpha\in\R,$ the interval $(\alpha,\alpha+l)\subset\R$ contains a number $\tau=\tau(\varepsilon)$ such that \eqref{eq: varepsilon-almost period} holds.
	\end{definition}
	
	For an almost periodic function $f,$ we denote its mean value as
	\[
	\mathfrak{M}\{f\}:=\lim_{T\rightarrow\infty}\frac{1}{2T}\int_{-T}^{T}f(y)dy.
	\]
	For every almost periodic function, its mean value always exists. The \emph{Bohr transform} $a(\lambda;f)$ of $f$ for $\lambda\in\R$ is defined as
	\begin{equation}
		a(\lambda;f):=\lim_{T\rightarrow\infty}\frac{1}{2T}\int_{-T}^{T}f(y)e^{-i\lambda y}dy=\mathfrak{M}\{f(y)e^{-i\lambda y}\},
	\end{equation}
	which is the analogue of the Fourier transform in the setting of almost periodic functions. Define the \emph{spectrum} of $f$ as
	\[
	\sigma(f):=\{\lambda\in\R:a(\lambda;f)=\mathfrak{M}\{f(y)e^{-i\lambda y}\}\neq0\},
	\]
	which is an at most countable set. Denote by $\{\lambda_k\}_{k=1}^{\infty}$ the spectrum of $f$, and let $a_k:=a(\lambda_k;f),$ which is referred to as the \emph{Bohr--Fourier coefficient} of $f.$ Then $f$ can be formally expanded as the \emph{Bohr--Fourier series}:
	\begin{equation}
		\label{eq: Bohr-Fourier series}
		f(y)\sim\sum_{k}a_ke^{i\lambda_ky}.
	\end{equation}
	The fundamental theorem in this area is \emph{Bohr's Approximation Theorem}, which says that every almost periodic function $f$ can be uniformly approximated by trigonometric polynomials; conversely, if a sequence of trigonometric polynomials converges uniformly, its limit function is necessarily almost periodic. The uniqueness theorem states that if $f$ and $g$ are two almost periodic functions with $a(\lambda;f)=a(\lambda;g)$ for every $\lambda\in\R,$ then $f\equiv g.$ When $X$ is a Hilbert space, there also exists a Parseval-type identity, namely, when \eqref{eq: Bohr-Fourier series} holds, then
	\begin{equation}
		\label{eq: Parseval identity}
		\mathfrak{M}\{\norm{f(y)}^2\}=\sum_{k}\norm{a_k}^2<\infty.
	\end{equation}
	
	Now we regard the cross-section of a periodic tube as a manifold. Let $(M_x,g_x)$ be a compact Riemannian manifold with Lipschitz boundary. When we consider a function $u:M_x\times\R_y\to\C$ on the periodic tube, by the definition of periodic tubes, $u$ may have the following invariance property:
	\begin{equation}
		\label{eq: functions with invariance property}
		u(x,y+L)=u(\varphi(x),y),\quad (x,y)\in M_x\times\R_y,
	\end{equation}
	where $L>0$ is a positive number (length) and $\varphi:M_x\rightarrow M_x$ is an isometry. The invariance \eqref{eq: functions with invariance property} can be interpreted as a boundary condition for $u$ on $M_x\times[0,L]_y:$
	\begin{equation}
		\label{eq: almost periodic boundary condition}
		u(x,L)=u(\varphi(x),0),\quad x\in M_x.
	\end{equation}
	When $\varphi=\mathrm{id},$ \eqref{eq: almost periodic boundary condition} is just the periodic boundary condition. In order to address the desired resolvent estimate, we first need to impose a condition on the isometry $\varphi$ that generalizes the periodic case.
	
	We call an isometry $\varphi:M_x\to M_x$ \emph{admissible} (the induced isometry on the boundary is denoted by $\varphi|_{\partial M_x}$) if for any $\varepsilon>0,$ the set
	\[
	S(\varphi,\varepsilon)=\{k\in\Z\mid \operatorname{dist}(\varphi^{k},\mathrm{id})<\varepsilon\}
	\]
	is \emph{relatively dense} in $\Z,$ where $\operatorname{dist}(\cdot,\cdot)$ denotes the distance between mappings in $C^{\infty}(M_x;M_x).$ Here a set $A\subset\Z$ is called relatively dense if there exists an $N\in\N$ such that every set of $N$ consecutive integers contains an element of $A.$ Now we give the definition of almost periodic boundary condition:
	\begin{definition}
		We call the boundary condition in \eqref{eq: almost periodic boundary condition} \emph{almost periodic} if the isometry $\varphi$ is admissible.
	\end{definition}
	One question is: which isometries $\varphi$ are admissible? In \cite{CGM}, the authors proved the following result:
	\begin{proposition}{\cite[Proposition~2.4]{CGM}}
		\label{prop: isometry is admissible}
		Let $\varphi:M_x\rightarrow M_x$ be an isometry, then $\varphi$ is admissible.
	\end{proposition}
	In particular, when $\varphi$ is a rotation, it is admissible, so the boundary condition \eqref{eq: almost periodic boundary condition} is almost periodic. The relation between almost periodic functions and admissible isometries is given by:
	\begin{lemma}{\cite[Lemma~2.3]{CGM}}
		Let $\varphi:M_x\rightarrow M_x$ be an admissible isometry and $u:M_x\times\R_y\to\C$ be such that $u\in C(\R_y;H^s(M_x))$ for some $s\in\R,$ satisfying that $u(x,y+L)=u(\varphi(x),y)$ for all $(x,y)$ and some $L>0$ fixed. Then the map
		\[g:\R_y\ni y\mapsto u(\cdot,y)\in H^s(M_x)\]
		is almost periodic.
	\end{lemma}
	
	\begin{remark}
		As stated in \cite{CGM}, the almost periodic boundary condition mentioned above is related to periodic tubes defined in the previous subsection as follows. Let $u\in\mathcal{C}(D_0)$ and take a periodic tube $F:\Omega_x\times\R_y\rightarrow D_0$ of length $L$ and associated rotation $\mathscr{R}.$ Consider the pullback $F^*u$ to $\Omega_x\times\R_y.$ By the definition of periodic tubes, we have
		\[F^*u(x,y+L)=F^*u(\mathscr{R}x,y),\quad(x,y)\in\Omega_x\times\R_y,\]
		which by Proposition~\ref{prop: isometry is admissible} means that $F^*u$ satisfies the almost periodic boundary condition \eqref{eq: almost periodic boundary condition}.
	\end{remark}
	
	\section{An Observability Resolvent Estimate}
	\label{sec: An Observability Resolvent Estimate}
	In this section we present an observability resolvent estimate which will be used later in the proof of Theorem~\ref{thm1.1}. Its proof uses semiclassical measures and relies closely on the dynamical properties of the billiard flow in polyhedra. Our proof essentially follows the line of proof in \cite{CGM}.
	\begin{proposition}
		\label{thm: obsr resolv estim}
		Let $P$ be a convex polyhedron in $\R^d$, and let $U$ be an arbitrary nonempty open neighborhood of the singular set $\mathcal{S}$ inside $P$. Then there exists $\lambda_0>0$ such that for any $\lambda\in\R$ with $|\lambda|>\lambda_0$ and any solution of
		\begin{equation}
			\label{eq: Helmholtz equation}
			(-\Delta-\lambda)u = f,\quad u|_{\partial P}=0,
		\end{equation}
		where $f\in H^s(P)\,(s>d/2)$, there exists a constant $C = C(U) > 0$ such that the following observability resolvent estimate holds:
		\begin{equation}
			\label{eq: obsr resolv estim}
			\norm{u}_{L^2(P)}\leq C\left(\norm{u}_{L^2(U)}+\norm{f}_{L^2(P)}\right).
		\end{equation}
	\end{proposition}
	\begin{remark}
		\label{remark: why s>d/2}
		Here we require that $f\in H^s(P)\,(s>d/2)$ so that, by the Sobolev embedding theorem, $f$ and hence $u$ are continuous in the spatial variables. This is because we will use a cutoff function to obtain a stationary equation on a periodic tube, where functions on the periodic tube are almost periodic along the direction of the tube (Definition~\ref{def: almost periodic function} is Bohr's definition of almost periodic function, which requires the function to be continuous with respect to its variable), and thus can be expanded into Bohr--Fourier series, so as to obtain the control result on the periodic tube using dimension reduction.
	\end{remark}
	\begin{proof}
		Recall that $D=(P\cup\sigma P)/{\sim}$ is the double of $P$ (see Section~\ref{sec: Billiard Dynamics on Polyhedra}). To facilitate our subsequent analysis, we extend $u$ and $f$ oddly to $\sigma P$ (the extended functions are still denoted by $u$ and $f$) so that they still satisfy the equation \eqref{eq: Helmholtz equation} on $D$ (the boundary condition now becomes $u|_{\partial D}=0$). By a slight abuse of notation, we regard $U$ as a neighborhood of the singular set $\mathcal{S}$ in both $D$ and $P.$ We will prove on $D$ the same result as \eqref{eq: obsr resolv estim} (with $P$ replaced by $D$), which automatically implies the conclusion of this proposition.
		
		We prove by contradiction. If \eqref{eq: obsr resolv estim} were false, there would exist sequences $\lambda_n\to\infty$ and $u_n,f_n$ satisfying
		\begin{equation}
			(-\Delta-\lambda_n)u_n=f_n
		\end{equation}
		such that
		\begin{equation}
			\label{eq: vanishing assumption}
			\norm{u_n}_{L^2(D)}=1,\quad \norm{u_n}_{L^2(U)}\to0,\quad \norm{f_n}_{L^2(D)}\to0.
		\end{equation}
		Now we only consider the case $\lambda_n>0$; the argument for $\lambda_n<0$ is exactly the same up to some sign changes. We introduce a semiclassical parameter $0<h_n\ll1$ by $h^2_n:=\frac{1}{\lambda_n}.$ Then we have
		\begin{equation}
			(-h^2_n\Delta-1)u_n=o_{L^2}(h_n).
		\end{equation}
		
		By \eqref{eq: vanishing assumption} we know that the sequence $(u_n)$ is bounded in $L^2(D),$ so there exists a semiclassical measure $\mu$ on $S^*D_0$ (where $D_0=D\setminus\mathcal{S}$) associated with $(u_n)$, possibly after extracting a subsequence, such that for any $a\in\mathcal{C}_0^{\infty}(S^*D_0)$ we have
		\begin{equation}
			\label{eq: def of semiclassical measure}
			\lim_{n\to\infty}\langle a(x,h_nD)u_n,u_n\rangle_{L^2(D_0)}=\int_{S^*D_0}a\,d\mu.
		\end{equation}
		Let $U_0:=U\setminus\mathcal{S}$ and denote by $\pi:S^*D_0\to D_0$ the base point projection. For $\mu$ we have the following lemma:
		\begin{lemma}
			\label{lemma: semiclassical measure}
			The support of $\mu$ is disjoint from $\pi^{-1}(U_0)$ and $\mu$ is a probability measure which is invariant under the geodesic flow.
		\end{lemma}
		\begin{proof}
			Suppose to the contrary that there is a point $q\in\operatorname{supp}\mu$ with $\pi(q)\in U_0$. Then there exists a small neighborhood $V$ of $q$ such that $V\subseteq U_0.$ Choose a nonnegative function $\phi\in\mathcal{C}^{\infty}(D_0)$ supported in $U_0$ with $\phi\equiv1$ in $V.$ Since $\mu$ is a nonnegative measure, we have $\langle\mu,\phi\rangle\geq0.$ If $\langle\mu,\phi\rangle=0,$ then for any $\chi\in\mathcal{C}_0^{\infty}(S^*D_0)$ supported in $\pi^{-1}(V),$ we have $\chi=\chi\phi,$ and hence $\langle\mu,\chi\rangle=\langle\mu,\chi\phi\rangle.$ By the nonnegativity of $\mu$ and $\phi,$ we have $|\langle\mu,\chi\phi\rangle|\leq\langle\mu,\phi\rangle\|\chi\|_{\infty}=0.$ This means that $\pi^{-1}(V)$ is disjoint from the support of $\mu,$ contradicting our earlier assumption; hence $\langle\mu,\phi\rangle>0$. This means that
			\[\lim\limits_{n\to\infty}\langle\phi u_n,u_n\rangle_{L^2(U_0)}>0,\]
			which contradicts our earlier assumption \eqref{eq: vanishing assumption}.
			
			To show that $\mu$ is a probability measure, we can take a cutoff function equal to $1$ on $D\setminus U_0$. Since we have already shown that the support of $\mu$ is disjoint from $\pi^{-1}(U_0),$ using the definition of the semiclassical measure \eqref{eq: def of semiclassical measure} together with the assumption that $u_n$ is normalized (see \eqref{eq: vanishing assumption}) we obtain the result. Finally, the invariance of $\mu$ under the geodesic flow is a standard property of semiclassical measures (see for example \cite[Section~5.2]{Zwor12}).
		\end{proof}
		\begin{remark}
			\label{what invariance means}
			The invariance of $\mu$ under the geodesic flow implies that if $(z,\xi)\in\operatorname{supp}\mu,$ then $(z+t\xi,\xi)\in\operatorname{supp}\mu$ for all $t\in\R.$
		\end{remark}
		
		Now we continue the proof of Proposition~\ref{thm: obsr resolv estim}. By Lemma~\ref{lemma: semiclassical measure}, what remains is to prove that $\mu$ vanishes on $S^*D_0$, which would contradict the fact that $\mu$ is a probability measure.
		
		As in \cite{CGM}, before proving the general case, we first present the proof for the case of rational polyhedra, which is easier and contains some preliminary ideas.
		
		\noindent\emph{Proof of Proposition~\ref{thm: obsr resolv estim} for Rational Polyhedra.}
		Let $(z_0,\xi_0)$ be in the support of $\mu.$ By Lemma~\ref{lemma: semiclassical measure} and Remark~\ref{what invariance means} we know that the billiard trajectory determined by $(z_0,\xi_0)$ is disjoint from $U_0,$ so by Corollary~\ref{thm: dichotomy} we can get that $z_0$ belongs to a maximal periodic tube in the direction of $\xi_0.$ Thus the finite tube condition (see Proposition~\ref{thm: finite tube condition}) gives that the support of $\mu$ is contained in the union of finitely many lifted maximal periodic tubes $\bigcup_{i=1}^{N}\mathscr{C}_i.$ This means that there are only finitely many directions in the support of $\mu.$ Without loss of generality, we assume that $U$ is the $\varepsilon$-neighbourhood $U_{\varepsilon}$ of $\mathcal{S}$ for an $\varepsilon>0.$
		
		Choose one such tube $\mathscr{C}:=\mathscr{C}_i.$ For a rational polyhedron, by Proposition~\ref{thm: w periodic} we know that every geodesic in $\mathscr{C}$ is periodic, and we can find an $L>0$ such that all geodesics in $\mathscr{C}$ have period $L.$ By the definition of periodic tubes, $\mathscr{C}$ is the image under a local isometry $F:\Omega\times\R\rightarrow D,$ where $\Omega$ is a convex polygon. Using the local isometry $F$ we can pull back $u_n$ to $\mathscr{C}.$ Denote the coordinates on $\Omega$ by $x$ and on $\R$ by $y.$ Let $(\partial\Omega)^{\varepsilon}:=\{x\in\Omega\mid d(x,\partial\Omega)\leq\varepsilon\}$ (where $d(\cdot,\cdot)$ is the distance in $\Omega$) be the $\varepsilon$-neighborhood of $\partial\Omega$ inside $\Omega,$ and let $\Omega_{\varepsilon}:=\Omega\setminus(\partial\Omega)^{\varepsilon}.$ By the definition of $\Omega_{\varepsilon}$ we can see that $\Omega_{\varepsilon}\times\R$ does not intersect $F^{-1}(U_{\varepsilon})$ due to Proposition~\ref{thm: boundary arbitrarily close to singular set}. Now choose a smooth cutoff function $\chi(x)$ such that $\chi=1$ inside $\Omega_{\varepsilon}$ and $\chi=0$ outside $\Omega_{\varepsilon/2}.$ Let $v_n:=\chi u_n;$ then the semiclassical measure $\nu$ associated with the sequence $(v_n)$ satisfies $\nu=\chi^2\mu.$ Since there are only finitely many directions in $\operatorname{supp}\mu$ and hence in $\operatorname{supp}\nu,$ we can find a constant-coefficient semiclassical pseudodifferential operator $\Phi=a(hD)$ with symbol $a(\xi)$ in $\R^n,$ such that $a(\xi)$ is microlocally equal to $1$ in a neighborhood of $\xi_0=dy$ (the direction of $\mathscr{C}$), but vanishes microlocally in a neighborhood of every other direction in $\operatorname{supp}\nu.$
		
		Let $\Phi_n:=a(h_nD).$ Now consider the sequence of functions $w_n:=\Phi_nv_n$ on $\mathscr{C}.$ The semiclassical measure $\nu'$ associated with the sequence $(w_n)$ is related to $\nu$ and $\mu$ by $\nu'=|a|^2\nu=|a|^2\chi^2\mu.$ Since we have assumed that $U=U_{\varepsilon}$ (the $\varepsilon$-neighborhood of the singular set $\mathcal{S}$) and that $\mu$ vanishes on $S^*U_0$ (see Lemma~\ref{lemma: semiclassical measure}), by our choices of $\Phi,\chi$ and the invariance of $\mu$ under the geodesic flow, together with Proposition~\ref{thm: boundary arbitrarily close to singular set}, we have that the support of $\nu'$ is restricted to directions parallel to $dy$ and to the geodesics parametrized by $x$ such that $\chi(x)=1$ (i.e., $\Omega_{\varepsilon}$).
		
		By a direct computation we know that $\Phi$ commutes with constant-coefficient differential operators; thus $w_n$ satisfies the following equation on $\Omega\times[0,L]$:
		\begin{equation}
			\label{eq: equ with cutoff}
			\begin{split}
				(-\Delta_{\Omega}-\partial^2_y-\lambda_n)w_n
				&=-\Phi_n((\Delta\chi)u_n)-2\Phi_n(\nabla_x\chi\cdot\nabla_xu_n)+\Phi_n(\chi f_n),\\
				w_n(x,0)
				&=w_n(x,L).
			\end{split}
		\end{equation}
		where the right-hand side of the equation in \eqref{eq: equ with cutoff} is just $[\Phi_n\chi,\Delta]u_n+\Phi_n(\chi f_n).$ Choose $\omega\subset\Omega$ to be a small enough neighborhood of $\partial\Omega$ such that $\omega$ is contained in the set $\{\chi=0\};$ for example, we can choose $\omega=\Omega\setminus\Omega_{\varepsilon/2}.$ Let $\varphi(x,y):=(x,y+L)$ and let $\mathscr{C}_{\varphi}:=\Omega\times[0,L]/(x,L)\sim(x,0)$ be the mapping torus determined by $\varphi.$ In addition, let $\omega_{\varphi}:=\omega\times[0,L]/(x,L)\sim(x,0)$ be defined in the same way. Now we can apply \cite[Theorem~5.4]{CGM} to obtain
		\[
		\norm{w_n}_{L^2(\mathscr{C}_{\varphi})}\leq C\left(\norm{w_n}_{L^2(\omega_{\varphi})}+\norm{-\Phi_n((\Delta\chi)u_n)-2\Phi_n(\nabla_x\chi\cdot\nabla_xu_n)+\Phi_n(\chi f_n)}_{H^{-1}_xL^2_y(\mathscr{C}_{\varphi})}\right).
		\]
		Since $\omega$ is contained in the set $\{\chi=0\},$ we have
		\[
		\norm{w_n}_{L^2(\omega_{\varphi})}\to0~(n\to\infty).
		\]
		Now we estimate the second term on the right-hand side of the above inequality. Let
		\[
		g_n:=\Phi_n((\Delta\chi)u_n)+2\Phi_n(\nabla_x\chi\cdot\nabla_xu_n).
		\]
		Since $\Phi_n$ commutes with constant-coefficient differential operators, we have
		\begin{equation*}
			\begin{split}
				\norm{g_n}_{H^{-1}_xL^2_y(\mathscr{C}_{\varphi})}
				&=\norm{\Phi_n((\Delta\chi)u_n)+2\Phi_n(\nabla_x\chi\cdot\nabla_xu_n)}_{H^{-1}_xL^2_y(\mathscr{C}_{\varphi})}\\
				&\leq C\left(\norm{\Phi_n((\Delta\chi)u_n)}_{H^{-1}_xL^2_y(\mathscr{C}_{\varphi})}+\norm{\nabla_x\cdot(\Phi_n(u_n\nabla_x\chi))}_{H^{-1}_xL^2_y(\mathscr{C}_{\varphi})}\right)\\
				&\leq C\left(\norm{\Phi_n((\Delta\chi)u_n)}_{L^2_{x,y}(\mathscr{C}_{\varphi})}+\norm{\Phi_n(u_n\nabla_x\chi)}_{L^2_{x,y}(\mathscr{C}_{\varphi})}\right).
			\end{split}
		\end{equation*}
		We can deduce that the two terms on the right-hand side of the above inequality both tend to $0$ as $n\to\infty$, since the cutoff $\chi$ we chose satisfies $\operatorname{supp}(\nabla\chi)\subset(\Omega_{\varepsilon/2}\setminus\Omega_{\varepsilon})$, but the semiclassical measure $\nu'$ vanishes on $(\Omega_{\varepsilon/2}\setminus\Omega_{\varepsilon})\times\R$. Thus we have
		\[
		\norm{g_n}_{H^{-1}_xL^2_y(\mathscr{C}_{\varphi})}\to0~(n\to\infty).
		\]
		Moreover, by \eqref{eq: vanishing assumption} we have
		\[
		\norm{\Phi_n(\chi f_n)}_{H^{-1}_xL^2_y(\mathscr{C}_{\varphi})}\leq\norm{\Phi_n(\chi f_n)}_{L^2_{x,y}(\mathscr{C}_{\varphi})}\to0~(n\to\infty).
		\]
		Combining the above estimates, we have
		\[
		\norm{w_n}_{L^2(\mathscr{C}_{\varphi})}\leq C\left(\norm{w_n}_{L^2(\omega_{\varphi})}+\norm{g_n}_{H^{-1}_xL^2_y(\mathscr{C}_{\varphi})}+\norm{\Phi_n(\chi f_n)}_{H^{-1}_xL^2_y(\mathscr{C}_{\varphi})}\right)\to0~(n\to\infty).
		\]
		This implies that $\nu',$ and hence $\nu,$ does not have any mass on $\mathscr{C}$ in the direction of $dy.$ Applying the above argument to each periodic tube $\mathscr{C}_i$ we can obtain that $\mu \equiv 0$, since the choice of $(z_0,\xi_0)\in\operatorname{supp}\mu$ was arbitrary. This contradicts the fact that $\mu$ is a probability measure.
		
		Now we give the proof in the general case, that is, the case of irrational polyhedra. In contrast with the rational case, for an irrational polyhedron, trajectories in a periodic tube need not be individually periodic. Consequently, the periodic boundary condition in \eqref{eq: equ with cutoff} will be replaced by an almost periodic boundary condition, and our proof will rely on the Property~(P2').
		
		\noindent\emph{Proof of Proposition~\ref{thm: obsr resolv estim} for General Case.}
		As before, let $\mu$ be the semiclassical measure associated with the sequence $(u_n),$ and let $(z_0,\xi_0)$ be in the support of $\mu.$ Using the same argument as in the case of rational polyhedra, we can get that the trajectory of $(z_0,\xi_0)$ under the geodesic flow is contained in a maximal periodic tube $\mathscr{C}$ in the direction of $\xi_0.$ The difference is that the shape of the cross-section of the maximal periodic tube changes. Assume that the maximal periodic tube $\mathscr{C}$ is of length $L,$ with associated rotation $\mathscr{R}.$ By definition, there is a local isometry $F:\Omega\times\R\rightarrow \mathscr{C},$ where $\Omega$ is a convex set invariant under $\mathscr{R}.$ Using this local isometry, by a slight abuse of notation, we identify $\Omega\times\R$ with $\mathscr{C}$ and $F^*u_n$ with $u_n;$ we also pull back the measure $\mu$ to $\mathscr{C}.$ Still denote the coordinates on $\Omega$ by $x$ and on $\R$ by $y.$ Note that by the definition of periodic tubes (see Definition~\ref{def: periodic tube}), we have $F\circ\varphi=F,$ where $\varphi(x,y):=(\mathscr{R}^{-1}x,y+L);$ thus $\mu$ is invariant under $\varphi.$ Recall that $(\partial\Omega)^{\varepsilon}=\{x\in\Omega\mid d(x,\partial\Omega)\leq\varepsilon\}$ and $\Omega_{\varepsilon}=\Omega\setminus(\partial\Omega)^{\varepsilon}.$
		
		Choose $\chi\in\mathcal{C}_0^{\infty}(\Omega)$ invariant under $\mathscr{R}$ and such that
		\[
		\chi(x)=\begin{cases}
			1,&x\in\Omega_{\varepsilon/2+\eta},\\
			0,&x\in\Omega\setminus\Omega_{\varepsilon/2},
		\end{cases}
		\]
		where $\eta$ is a sufficiently small positive constant. Let $v_n:=\chi u_n.$ Then the semiclassical measure associated with the sequence $(v_n)$ is $\nu=\chi^2\mu.$
		
		Take now $\Phi=a(hD),$ a constant-coefficient semiclassical $\Psi$DO with symbol $a(\xi)$ in $\R^n,$ microlocally cutting off near $dy$ and constructed as follows. Using the Property~(P2'), we may take $a(\xi)$ supported in a small cone $\Gamma\subset\R^n\setminus0$ around $dy,$ such that all lines in the direction of $\Gamma$ with basepoint $x\in\Omega_{\varepsilon/2-\eta}\setminus\Omega_{\varepsilon/2+\eta}$ hit the set $F^{-1}(U_{\varepsilon})$ in finite time, where $U_{\varepsilon}$ is the $\varepsilon$-neighborhood of the singular set $\mathcal{S}$. Moreover, we may take $a(\xi)$ equal to $1$ near $dy$ and invariant under rotations around $dy,$ i.e., $a\circ\mathscr{R}=a.$
		
		Let $w_n:=\Phi_nv_n,$ where $\Phi_n=a(h_nD).$ By \cite[Proposition~2.5]{CGM} we know that $w_n$ is invariant under $\varphi$ if $v_n$ is invariant under $\varphi$ and $a\circ\mathscr{R}=a,$ and consequently we can pull back $w_n$ to $\mathscr{C}_{\varphi}$ where $\mathscr{C}_{\varphi}:=\Omega\times[0,L]/(x,L)\sim(\mathscr{R}x,0)$ is the mapping torus determined by $\varphi.$ The semiclassical measure $\nu'$ associated with $(w_n)$ also satisfies $\nu'=|a|^2\nu=|a|^2\chi^2\mu.$ Let $U=U_{\varepsilon}$ again. Since $\mu$ vanishes on $F^{-1}(U_{\varepsilon}),$ by our choices of $\Phi,\chi$ and the invariance of $\mu$ under the geodesic flow, together with the description in the previous paragraph, we have that $\nu'=0$ on $(\Omega\setminus\Omega_{\varepsilon/2+\eta})\times\R.$ By \cite[Proposition~2.5]{CGM} again we know that $\Phi_n$ commutes with constant-coefficient differential operators on the periodic tube; thus $w_n$ satisfies the following equation on $\Omega\times[0,L]$ with an almost periodic boundary condition:
		\begin{equation}
			\label{eq: equ with cutoff in irration}
			\begin{split}
				(-\Delta_{\Omega}-\partial^2_y-\lambda_n)w_n
				&=-\Phi_n((\Delta\chi)u_n)-2\Phi_n(\nabla_x\chi\cdot\nabla_xu_n)+\Phi_n(\chi f_n),\\
				w_n(\mathscr{R}x,0)
				&=w_n(x,L).
			\end{split}
		\end{equation}
		
		Choose $\omega\subset\Omega$ to be a small enough neighborhood of $\partial\Omega$ such that $\omega$ is contained in the set $\{\chi=0\}$ and is invariant under $\mathscr{R};$ for example $\omega=\Omega\setminus\Omega_{\varepsilon/2-\eta}$ would work. We have defined $\mathscr{C}_{\varphi}$ above; now let $\omega_{\varphi}:=\omega\times[0,L]/(x,L)\sim(\mathscr{R}x,0)$. Using \cite[Theorem~5.4]{CGM} again, we can still obtain
		\[
		\norm{w_n}_{L^2(\mathscr{C}_{\varphi})}\leq C\left(\norm{w_n}_{L^2(\omega_{\varphi})}+\norm{-\Phi_n((\Delta\chi)u_n)-2\Phi_n(\nabla_x\chi\cdot\nabla_xu_n)+\Phi_n(\chi f_n)}_{H^{-1}_xL^2_y(\mathscr{C}_{\varphi})}\right).
		\]
		Let
		\[
		g_n:=\Phi_n((\Delta\chi)u_n)+2\Phi_n(\nabla_x\chi\cdot\nabla_xu_n),
		\]
		then
		\[
		\norm{w_n}_{L^2(\mathscr{C}_{\varphi})}\leq C\left(\norm{w_n}_{L^2(\omega_{\varphi})}+\norm{g_n}_{H^{-1}_xL^2_y(\mathscr{C}_{\varphi})}+\norm{\Phi_n(\chi f_n)}_{H^{-1}_xL^2_y(\mathscr{C}_{\varphi})}\right).
		\]
		Using the same argument as in the rational polyhedra case together with the previous assumptions, we have
		\[
		\norm{w_n}_{L^2(\omega_{\varphi})}\to0~(n\to\infty)
		\]
		and
		\[
		\norm{\Phi_n(\chi f_n)}_{H^{-1}_xL^2_y(\mathscr{C}_{\varphi})}\leq\norm{\Phi_n(\chi f_n)}_{L^2_{x,y}(\mathscr{C}_{\varphi})}\to0~(n\to\infty).
		\]
		In addition, since $\Phi_n$ commutes with constant-coefficient differential operators, we have
		\[
		\norm{g_n}_{H^{-1}_xL^2_y(\mathscr{C}_{\varphi})}\leq C\left(\norm{\Phi_n((\Delta\chi)u_n)}_{L^2_{x,y}(\mathscr{C}_{\varphi})}+\norm{\Phi_n(u_n\nabla_x\chi)}_{L^2_{x,y}(\mathscr{C}_{\varphi})}\right).
		\]
		We have already shown that $\nu'=0$ on $(\Omega_{\varepsilon/2}\setminus\Omega_{\varepsilon/2+\eta})\times\mathbb{R}$, and $\operatorname{supp}(\nabla\chi)\subset\Omega_{\varepsilon/2}\setminus\Omega_{\varepsilon/2+\eta}$. Hence the two terms on the right-hand side of the above expression both tend to $0$ as $n\to\infty$, and consequently
		\[
		\norm{g_n}_{H^{-1}_x L^2_y(\mathscr{C}_{\varphi})}\to0~(n\to\infty).
		\]
		Thus we obtain
		\[
		\norm{w_n}_{L^2(\mathscr{C}_{\varphi})}\leq C\left(\norm{w_n}_{L^2(\omega_{\varphi})}+\norm{g_n}_{H^{-1}_xL^2_y(\mathscr{C}_{\varphi})}+\norm{\Phi_n(\chi f_n)}_{H^{-1}_xL^2_y(\mathscr{C}_{\varphi})}\right)\to0~(n\to\infty).
		\]
		This implies that $\nu'\equiv0,$ and further implies that $\mu\equiv0$ since the choice of $(z_0,\xi_0)\in\operatorname{supp}\mu$ was arbitrary. This contradicts the fact that $\mu$ is a probability measure. We now complete the proof.
	\end{proof}
	
	\section{Proof of Theorem~\ref{thm1.1}}
	\label{sec: Proof of Theorem1.1}
	In this section we give the proof of Theorem~\ref{thm1.1} using the resolvent estimate \eqref{eq: obsr resolv estim}. The proof essentially follows that of \cite[Theorem~7]{BZ04} (see also \cite{J18}); we apply the argument given there in an abstract setting to the concrete situation here.
	
	\begin{proof}[Proof of Theorem~\ref{thm1.1}]
		First, let $v(t)=e^{it\Delta}u_0$ be the solution of the homogeneous Schr\"{o}dinger equation with initial value $u_0$. Introduce a function $\psi(t)\in\mathcal{C}_0^{\infty}(\R_t)$ such that $\operatorname{supp}(\psi)\subset(0,1)$ and set $w(t)=\psi\left(\frac{t}{T}\right)v(t)$ where $T>0$ is a constant to be chosen. Then $w(t)$ satisfies the following equation:
		\begin{equation}
			(i\partial_t+\Delta)w(t)=\frac{i}{T}\psi'\left(\frac{t}{T}\right)v(t).
		\end{equation}
		Since $w(t)$ is compactly supported in $t$, taking the (adjoint) Fourier transform of the above equation with respect to $t$ yields
		\begin{equation}
			(-\Delta-\tau)\mathcal{F}_{t\to\tau}^*w(\tau)=-\frac{i}{T}\mathcal{F}_{t\to\tau}^*(\psi'(\cdot/T)v)(\tau).
		\end{equation}
		Here
		\[
		\mathcal{F}_{t\to\tau}^*w(\tau)=\int_{\R_t}e^{it\tau}w(t)dt
		\]
		denotes the adjoint Fourier transform of $w(t)$ with respect to $t$, and the same applies to $\mathcal{F}_{t\to\tau}^*(\psi'(\cdot/T)v)(\tau).$ Plancherel's theorem gives
		\[
		\norm{\mathcal{F}^*_{t\to\tau}w}^2_{L^2(\R_{\tau})}=2\pi\norm{w}^2_{L^2(\R_t)},\quad
		\norm{\mathcal{F}^*_{t\to\tau}(\psi'(\cdot/T)v)}^2_{L^2(\R_{\tau})}=2\pi\norm{\psi'(\cdot/T)v}^2_{L^2(\R_t)}.
		\]
		
		Let $\rho>0$ be a large constant such that for any $\langle\tau\rangle\geq\rho/2$, the $\mathcal{F}_{t\to\tau}^*w(\tau)$ satisfies \eqref{eq: obsr resolv estim} in Proposition~\ref{thm: obsr resolv estim}. For $\langle\tau\rangle\geq\rho/2,$ we estimate $\mathcal{F}_{t\to\tau}^*w(\tau)$ using \eqref{eq: obsr resolv estim} which gives
		\begin{equation}
			\label{eq: estimate for large tau}
			\norm{\mathcal{F}_{t\to\tau}^*w(\tau)}_{L^2(P)}\leq C\left(\frac{1}{T}\norm{\mathcal{F}_{t\to\tau}^*(\psi'(\cdot/T)v)(\tau)}_{L^2(P)}+\norm{\mathcal{F}_{t\to\tau}^*w(\tau)}_{L^2(U)}\right).
		\end{equation}
		For $\langle\tau\rangle<\rho/2,$ we choose $\chi\in\mathcal{C}_0^{\infty}((-1,1);[0,1])$ equal to $1$ on $[-1/2,1/2]$ and split $\mathcal{F}_{t\to\tau}^*w(\tau)$ into two parts by frequency truncation on $u_0$:
		\begin{equation}
			\begin{split}
				\mathcal{F}_{t\to\tau}^*w(\tau)
				&=\int_{\R_t}e^{it(\Delta+\tau)}\psi\left(\frac{t}{T}\right)(\chi(\Delta/\rho)u_0+(1-\chi(\Delta/\rho))u_0)\,dt\\
				&=J_1(\tau)+J_2(\tau),
			\end{split}
		\end{equation}
		where
		\begin{equation}
			J_1(\tau):=\int_{\R_t}e^{it(\Delta+\tau)}\psi\left(\frac{t}{T}\right)\chi(\Delta/\rho)u_0\,dt
		\end{equation}
		and
		\begin{equation}
			J_2(\tau):=\int_{\R_t}e^{it(\Delta+\tau)}\psi\left(\frac{t}{T}\right)(1-\chi(\Delta/\rho))u_0\,dt.
		\end{equation}
		Let
		\[
		I^{\rho}_{\tau}:=\{\tau:\langle\tau\rangle<\rho/2\}.
		\]
		For $J_1(\tau)$, by an application of Plancherel's theorem, we have
		\begin{equation}
			\label{eq: estimate of J_1}
			\begin{split}
				\norm{J_1(\tau)}^2_{L^2(I^{\rho}_{\tau},L^2(P))}
				&\leq\norm{J_1(\tau)}^2_{L^2(\R_{\tau},L^2(P))}\\
				&=2\pi\int_{\R_t}\psi\left(\frac{t}{T}\right)^2\norm{e^{it\Delta}\chi(\Delta/\rho)u_0}^2_{L^2(P)}dt\\
				&=2\pi T\left(\int_{\R_t}\psi(t)^2dt\right)\norm{\chi(\Delta/\rho)u_0}^2_{L^2(P)}\\
				&=2\pi T\norm{\psi(t)}^2_{L^2(\R_t)}\norm{\chi(\Delta/\rho)u_0}^2_{L^2(P)}.
			\end{split}
		\end{equation}
		For $J_2(\tau),$ we use the fact that for any positive integer $N,$
		\[
		e^{it(\Delta+\tau)}=(i(\Delta+\tau))^{-N}\partial^N_te^{it(\Delta+\tau)},
		\]
		so that by integrating by parts $N$ times we obtain
		\begin{equation}
			\begin{split}
				J_2(\tau)
				&=\int_{\R_t}e^{it(\Delta+\tau)}\psi\left(\frac{t}{T}\right)(1-\chi(\Delta/\rho))u_0\,dt\\
				&=\left(\frac{-1}{T}\right)^{N}\int_{\R_t}(i(\Delta+\tau))^{-N}e^{it(\Delta+\tau)}(1-\chi(\Delta/\rho))u_0\psi^{(N)}\left(\frac{t}{T}\right)dt.
			\end{split}
		\end{equation}
		We want to mention that the cutoff $\chi$ we chose ensures that in the above expression, when $\langle\tau\rangle<\rho/2$, the operator $(\Delta+\tau)^{-N}$ acting on $(1-\chi(\Delta/\rho))u_0$ is well-defined.
		Thus we have
		\begin{equation}
			\begin{split}
				\left|J_2(\tau)\right|
				&\leq\frac{1}{T^N}\int_{\R_t}\left|(\Delta+\tau)^{-N}(1-\chi(\Delta/\rho))u_0\psi^{(N)}\left(\frac{t}{T}\right)\right|dt\\
				&=\frac{1}{T^N}\left|(\Delta+\tau)^{-N}(1-\chi(\Delta/\rho))u_0\right|\int_{\R_t}\left|\psi^{(N)}\left(\frac{t}{T}\right)\right|dt\\
				&\leq\frac{\max_{t\in\R_t}\left|\psi^{(N)}(t)\right|}{T^{N-1}}\left|(\Delta+\tau)^{-N}(1-\chi(\Delta/\rho))u_0\right|.
			\end{split}
		\end{equation}
		Note that by functional calculus we have
		\[
		\norm{(\Delta+\tau)^{-N}(1-\chi(\Delta/\rho))u_0}_{L^2(P)}\leq C_N\langle\tau\rangle^{-N}\norm{(1-\chi(\Delta/\rho))u_0}_{L^2(P)}.
		\]
		Using this estimate we can get
		\begin{equation}
			\begin{split}
				\norm{J_2(\tau)}_{L^2(P)}
				&\leq\frac{C_N\max_{t\in\R_t}\left|\psi^{(N)}(t)\right|}{T^{N-1}}\langle\tau\rangle^{-N}\norm{(1-\chi(\Delta/\rho))u_0}_{L^2(P)}\\
				&\leq\frac{C_N\max_{t\in\R_t}\left|\psi^{(N)}(t)\right|}{T^{N-1}}\langle\tau\rangle^{-N}\norm{u_0}_{L^2(P)}.
			\end{split}
		\end{equation}
		Thus we have
		\begin{equation}
			\label{eq: estimate of J_2}
			\begin{split}
				\norm{J_2(\tau)}_{L^2(I^{\rho}_{\tau},L^2(P))}
				&\leq\frac{C_N\max_{t\in\R_t}\left|\psi^{(N)}(t)\right|}{T^{N-1}}\norm{\langle\tau\rangle^{-N}}_{L^2(\R_{\tau})}\norm{u_0}_{L^2(P)}\\
				&\leq\frac{C'_N\max_{t\in\R_t}\left|\psi^{(N)}(t)\right|}{T^{N-1}}\norm{u_0}_{L^2(P)}.
			\end{split}
		\end{equation}
		Now combining \eqref{eq: estimate for large tau}, \eqref{eq: estimate of J_1} and \eqref{eq: estimate of J_2} we can get
		\begin{equation}
			\label{eq: total estimate}
			\begin{split}
				\norm{\mathcal{F}^*w(\tau)}&^2_{L^2(\R_{\tau},L^2(P))}\\
				&\leq \frac{C}{T^2}\norm{\mathcal{F}_{t\to\tau}^*(\psi'(\cdot/T)v)(\tau)}^2_{L^2(\R_{\tau},L^2(P))}+C\norm{\mathcal{F}_{t\to\tau}^*w(\tau)}^2_{L^2(\R_{\tau},L^2(U))}\\
				&\quad+\norm{J_1(\tau)}^2_{L^2(I^{\rho}_{\tau},L^2(P))}+\norm{J_2(\tau)}^2_{L^2(I^{\rho}_{\tau},L^2(P))}\\
				&\leq \frac{C}{T^2}\norm{\mathcal{F}_{t\to\tau}^*(\psi'(\cdot/T)v)(\tau)}^2_{L^2(\R_{\tau},L^2(P))}+C\norm{\mathcal{F}_{t\to\tau}^*w(\tau)}^2_{L^2(\R_{\tau},L^2(U))}\\
				&\quad+2\pi T\norm{\psi(t)}^2_{L^2(\R_t)}\norm{\chi(\Delta/\rho)u_0}^2_{L^2(P)}+\frac{C'_N\max_{t\in\R_t}\left|\psi^{(N)}(t)\right|^2}{T^{2(N-1)}}\norm{u_0}^2_{L^2(P)}.
			\end{split}
		\end{equation}
		Note that by Plancherel's theorem we have
		\begin{equation}
			\norm{\mathcal{F}^*_{t\to\tau}w(\tau)}^2_{L^2(\R_{\tau},L^2(P))}=2\pi T\norm{\psi(t)}^2_{L^2(\R_t)}\norm{u_0}^2_{L^2(P)}
		\end{equation}
		and
		\begin{equation}
			\norm{\mathcal{F}^*_{t\to\tau}(\psi'(\cdot/T)v)(\tau)}^2_{L^2(\R_{\tau},L^2(P))}=2\pi T\norm{\psi'(t)}^2_{L^2(\R_t)}\norm{u_0}_{L^2(P)}^2.
		\end{equation}
		Substituting the above two equalities into \eqref{eq: total estimate} we obtain
		\begin{equation}
			\begin{split}
				\label{eq: total estimate for u_0}
				\norm{u_0}^2_{L^2(P)}
				&\leq\frac{C\norm{\psi'(t)}^2_{L^2(\R_t)}}{T^2\norm{\psi(t)}^2_{L^2(\R_t)}}\norm{u_0}^2_{L^2(P)}+\frac{C\max_{t\in\R_t}|\psi(t)|^2}{T\norm{\psi(t)}^2_{L^2(\R_t)}}\int_{0}^{T}\norm{e^{it\Delta}u_0}^2_{L^2(U)}dt\\
				&\quad+\norm{\chi(\Delta/\rho)u_0}^2_{L^2(P)}+\frac{C'_N\max_{t\in\R_t}\left|\psi^{(N)}(t)\right|^2}{T^{2N-1}\norm{\psi(t)}^2_{L^2(\R_t)}}\norm{u_0}^2_{L^2(P)}.
			\end{split}
		\end{equation}
		Now, as long as $\psi\not\equiv0$ is fixed, there exists $T_0>0$ such that when $T>T_0$ the first and last terms on the right-hand side of \eqref{eq: total estimate for u_0} can be absorbed into the left-hand side, and hence we obtain
		\begin{equation}
			\norm{u_0}^2_{L^2(P)}\leq C\int_{0}^{T}\norm{e^{it\Delta}u_0}^2_{L^2(U)}dt+C\norm{\chi(\Delta/\rho)u_0}^2_{L^2(P)}.
		\end{equation}
		By a simple observation we see that there exists a constant $C_{\rho}>0$ such that
		\[
		\norm{\chi(\Delta/\rho)u_0}^2_{L^2(P)}\leq C_{\rho}\norm{u_0}^2_{H^{-2}(P)}.
		\]
		Thus we have
		\begin{equation}
			\label{eq: obsr with error}
			\norm{u_0}^2_{L^2(P)}\leq C\int_{0}^{T}\norm{e^{it\Delta}u_0}^2_{L^2(U)}dt+C_{\rho}\norm{u_0}^2_{H^{-2}(P)}.
		\end{equation}
		
		Now we use the compactness-uniqueness argument which originates from \cite{BLR92} (see also \cite{BZ12}\cite{J18}) to eliminate the last term.
		
		Let us fix $\delta\geq0$ and define
		\[N_{\delta}:=\{u_0\in L^2(P):e^{it\Delta}u_0\equiv0\text{ on }(0,T-\delta)\times U\}.\]
		For a $u_0\in N_0$, define
		\[v_{\varepsilon,0}:=\frac{1}{\varepsilon}(e^{i\varepsilon\Delta}-I)u_0.\]
		Then if $\varepsilon\leq\delta,$ we have $v_{\varepsilon,0}\in N_{\delta}.$ We will show that $(v_{\varepsilon,0})$ is a Cauchy sequence in $L^2(P)$ and therefore converges. To see this, let $\{(\lambda_j,e_j)\}_{j=1}^{\infty}$ be a complete orthonormal basis of $L^2(P)$ formed by normalized eigenfunctions of the Laplacian, with associated eigenvalues $\lambda_j$. We expand $u_0$ in terms of this basis:
		\[u_0=\sum_{j=1}^{\infty}u_{0,j}e_j.\]
		Then for $\alpha,\beta\in(0,T/2),$ since $v_{\alpha,0},v_{\beta,0}\in N_{T/2},$ we have by \eqref{eq: obsr with error} (with $T$ replaced by $T/2$),
		\begin{equation*}
			\begin{split}
				\norm{v_{\alpha,0}-v_{\beta,0}}_{L^2(P)}^2
				&\leq C_{\rho}\norm{v_{\alpha,0}-v_{\beta,0}}_{H^{-2}(P)}^2\\
				&\leq C_{\rho}\sum_{j=1}^{\infty}\left|\frac{e^{-i\alpha\lambda_j}-1}{\alpha}-\frac{e^{-i\beta\lambda_j}-1}{\beta}\right|^2(1+\lambda_j)^{-2}|u_{0,j}|^2\\
				&\leq C_{\rho}\sum_{j=1}^{\infty}|\alpha-\beta|^2\lambda_j^2(1+\lambda_j)^{-2}|u_{0,j}|^2\\
				&\leq C_{\rho}|\alpha-\beta|^2\norm{u_0}_{L^2(P)}^2.
			\end{split}
		\end{equation*}
		Therefore, $\lim_{\alpha,\beta\rightarrow0}\norm{v_{\alpha,0}-v_{\beta,0}}_{L^2(P)}=0,$ which means that $(v_{\varepsilon,0})$ is a Cauchy sequence in $L^2(P).$ Then there exists $v_0\in L^2(P)$ such that
		\[L^2\text{-}\lim_{\varepsilon\rightarrow0}v_{\varepsilon,0}=v_0.\]
		This limit is necessarily in $N_{\delta}$ for any $\delta>0,$ hence in $N_0.$ By the definition of $v_{\varepsilon,0}$ we can get that \[v_0=i\Delta u_0.\] Hence $N_0$ is an invariant subspace of $\Delta$ in $L^2(P).$ From \eqref{eq: obsr with error} we see that the $L^2(P)$-norm and the $H^{-2}(P)$-norm are equivalent in $N_0$. Hence, by Rellich's theorem, the unit ball in $N_0$ is compact, and consequently, by Riesz's lemma, $N_0$ is finite-dimensional. If $N_0\neq\{0\},$ then there must exist an eigenvector $w$ of $-\Delta$ in $N_0$:
		\[-\Delta w=\lambda_0 w,\quad w|_{U}=0.\]
		By the unique continuation property of second-order elliptic differential operators, we have $w \equiv 0$, which yields a contradiction. Thus $N_0 = \{0\}$.
		
		Finally, we complete the proof of Theorem~\ref{thm1.1} by contradiction. If \eqref{eq: obsr in P} does not hold, then we can find a sequence $(u_{n,0})$ in $H^s(P)$ such that
		\[\norm{u_{n,0}}_{L^2(P)}=1,\quad\int_{0}^{T}\norm{e^{it\Delta}u_{n,0}}^2_{L^2(U)}dt\to0~(n\to\infty).\]
		Then we can extract a subsequence $(u_{n_k,0})$ converging weakly in $L^2(P)$ (and hence strongly in $H^{-2}(P)$) to a limit $u_0\in N_0.$ By \eqref{eq: obsr with error} we have
		\[1=\norm{u_{n_k,0}}^2_{L^2(P)}\leq C\int_{0}^{T}\norm{e^{it\Delta}u_{n_k,0}}^2_{L^2(U)}dt+C_{\rho}\norm{u_{n_k,0}}_{H^{-2}(P)}^2.\]
		Letting $k\rightarrow\infty$, we obtain
		\[1=\norm{u_0}^2_{L^2(P)}\leq C_{\rho}\norm{u_0}_{H^{-2}(P)}^2.\]
		This means that $u_0 \not\equiv 0$ and $u_0 \in N_0$. But we have already proved that $N_0 = \{0\}$, which yields a contradiction and thus completes the proof.
	\end{proof}
	
	\section{From Observability to Controllability: HUM}
	\label{sec: From Observability to Control}
	Now we use the classical HUM method developed by Lions in \cite{Lions88} to show that Theorem~\ref{thm1.1} implies Theorem~\ref{thm1.2}.
	\begin{proof}[Proof of Theorem~\ref{thm1.2}]
		For the system
		\begin{equation}
			\label{eq: null control}
			\left\{
			\begin{aligned}
				(i\partial_{t}+\Delta)u(t,z)=\mathbf{1}_{[0,T]\times U}g(t,z),\\
				u(0,z)=u_0,\quad u(T,z)\equiv0,
			\end{aligned}
			\right.
		\end{equation}
		where $g\in L^2([0,T];H^s(P))$ and $u_0\in H^s(P)$ with $s>d/2,$ we define the operator $R$ as
		\[
		\begin{aligned}
			R:L^2([0,T];H^s(P))&\to H^s(P)\\
			g&\mapsto u_0.
		\end{aligned}
		\]
		From the definition of $R$ we can see that to prove Theorem~\ref{thm1.2}, it is enough to prove that $R$ is surjective. Now consider the adjoint system of \eqref{eq: null control}:
		\begin{equation}
			\label{eq: adjoint system}
			\left\{
			\begin{aligned}
				(i\partial_{t}+\Delta)v(t,z)&=0,\\
				v(0,z)&=v_0,
			\end{aligned}
			\right.
		\end{equation}
		where $v_0\in H^s(P),$ and define the operator $S$ as
		\[
		\begin{aligned}
			S:H^s(P)&\to L^2([0,T];H^s(P))\\
			v_0&\mapsto e^{it\Delta}v_0.
		\end{aligned}
		\]
		Then, using the equation in \eqref{eq: null control} together with integration by parts, we obtain
		\[
		\begin{aligned}
			\langle \mathbf{1}_{[0,T]\times U}g,Sv_0\rangle_{L^2([0,T]\times P)}
			&=\int_{0}^{T}\int_{P}(i\partial_{t}+\Delta)u\cdot\overline{v}\,dzdt\\
			&=i\int_{P}(u\overline{v})|_{t=0}^{t=T}\,dz+\int_{0}^{T}\int_{P}u\cdot(-i\partial_{t}+\Delta)\overline{v}\,dzdt.
		\end{aligned}
		\]
		Since $u(T,z)\equiv0$ and by the definition of $R$, we have
		\[
		i\int_{P}(u\overline{v})|_{t=0}^{t=T}\,dz=-i\langle Rg,v_0\rangle_{L^2(P)}.
		\]
		Moreover, by the equation in \eqref{eq: adjoint system} we know that $(-i\partial_{t}+\Delta)\overline{v}=0.$ Thus we have
		\begin{equation}
			\label{eq: dual equa}
			\langle \mathbf{1}_{[0,T]\times U}g,Sv_0\rangle_{L^2([0,T]\times P)}=-i\langle Rg,v_0\rangle_{L^2(P)}.
		\end{equation}
		Now define the operator \[K:=-iR\circ S.\] Then by \eqref{eq: dual equa} we have, for any $u_0\in H^s(P)\,(s>d/2),$
		\[
		\begin{aligned}
			\langle Ku_0,u_0\rangle_{L^2(P)}
			&=-i\langle R\circ Su_0,u_0\rangle_{L^2(P)}\\
			&=\langle \mathbf{1}_{[0,T]\times U}Su_0,Su_0\rangle_{L^2([0,T]\times P)}\\
			&=\norm{Su_0}_{L^2([0,T]\times U)}^2\geq C\norm{u_0}_{L^2(P)}^2,
		\end{aligned}
		\]
		where the last inequality holds due to Theorem~\ref{thm1.1}. This implies that $K$ is a positive definite bounded self-adjoint operator, and hence an isomorphism on $H^s(P)\,(s>d/2).$
		
		Now define the control operator \[L:=-iS\circ K^{-1}.\] It is easy to check that for any $u_0\in H^s(P)\,(s>d/2),$
		\[R\circ Lu_0=u_0.\] Thus, taking $g=Lu_0$ in equation \eqref{eq: null control}, we can get that the solution $u$ of equation \eqref{eq: null control} satisfies $u(T,z)\equiv0$ (i.e., null controllability holds), which completes the proof.
	\end{proof}


\begin{thebibliography}{99}
		\bibitem{AR12} N. Anantharaman and G. Rivi\`ere, Dispersion and controllability for the Schr\"odinger equation on negatively curved manifolds, Anal. PDE {\bf 5} (2012), no.~2, 313--338.
		
		\bibitem{AM14} N. Anantharaman and F. Maci\`a, Semiclassical measures for the Schr\"odinger equation on the torus, J. Eur. Math. Soc. (JEMS) {\bf 16} (2014), no.~6, 1253--1288.
		
		\bibitem{ALM16} N. Anantharaman, M. L\'eautaud and F. Maci\`a, Wigner measures and observability for the Schr\"odinger equation on the disk, Invent. Math. {\bf 206} (2016), no.~2, 485--599.
		
		\bibitem{BLR92} C.~W. Bardos, G. Lebeau and J. Rauch, Sharp sufficient conditions for the observation, control, and stabilization of waves from the boundary, SIAM J. Control Optim. {\bf 30} (1992), no.~5, 1024--1065.
		
		\bibitem{BBZ13} J. Bourgain, N. Burq and M. Zworski, Control for Schr\"odinger operators on 2-tori: rough potentials, J. Eur. Math. Soc. (JEMS) {\bf 15} (2013), no.~5, 1597--1628.
		
		\bibitem{BZ04} N. Burq and M. Zworski, Geometric control in the presence of a black box, J. Amer. Math. Soc. {\bf 17} (2004), no.~2, 443--471.
		
		\bibitem{BZ05} N. Burq and M. Zworski, Bouncing ball modes and quantum chaos, SIAM Rev. {\bf 47} (2005), no.~1, 43--49.
		
		\bibitem{BZ12} N. Burq and M. Zworski, Control for Schr\"odinger operators on tori, Math. Res. Lett. {\bf 19} (2012), no.~2, 309--324.
		
		\bibitem{BZ19} N. Burq and M. Zworski, Rough controls for Schr\"odinger operators on 2-tori, Ann. H. Lebesgue {\bf 2} (2019), 331--347.
		
		\bibitem{CGM} M. Ceki\'c, B. Georgiev and M. Mukherjee, Polyhedral billiards, eigenfunction concentration and almost periodic control, Comm. Math. Phys. {\bf 377} (2020), no.~3, 2451--2487.
		
		\bibitem{DG02} C.~I. Delman and G.~A. Gal\textquotesingle perin, Billiards with pockets: a separation principle and bound for the number of orbit types, Comm. Math. Phys. {\bf 230} (2002), no.~3, 463--483.
		
		\bibitem{DJ18} S. Dyatlov and L. Jin, Semiclassical measures on hyperbolic surfaces have full support, Acta Math. {\bf 220} (2018), no.~2, 297--339.
		
		\bibitem{DJN} S. Dyatlov, L. Jin and S. Nonnenmacher, Control of eigenfunctions on surfaces of variable curvature, J. Amer. Math. Soc. {\bf 35} (2022), no.~2, 361--465.
		
		\bibitem{GKT} G.~A. Gal\textquotesingle perin, T. Kr\"uger and S.~E. Troubetzkoy, Local instability of orbits in polygonal and polyhedral billiards, Comm. Math. Phys. {\bf 169} (1995), no.~3, 463--473.
		
		\bibitem{Haraux} A. Haraux, S\'eries lacunaires et contr\^ole semi-interne des vibrations d'une plaque rectangulaire, J. Math. Pures Appl. (9) {\bf 68} (1989), no.~4, 457--465 (1990).
		
		\bibitem{HHM09} A. Hassell, L. Hillairet and J.~L. Marzuola, Eigenfunction concentration for polygonal billiards, Comm. Partial Differential Equations {\bf 34} (2009), no.~4-6, 475--485.
		
		\bibitem{Jaffard} S. Jaffard, Contr\^ole interne exact des vibrations d'une plaque rectangulaire, Portugal. Math. {\bf 47} (1990), no.~4, 423--429.
		
		\bibitem{J18} L. Jin, Control for Schr\"odinger equation on hyperbolic surfaces, Math. Res. Lett. {\bf 25} (2018), no.~6, 1865--1877.
		
		\bibitem{Komornik} V. Komornik, On the exact internal controllability of a Petrowsky system, J. Math. Pures Appl. (9) {\bf 71} (1992), no.~4, 331--342.
		
		\bibitem{Laur14} C. Laurent, Internal control of the Schr\"odinger equation, Math. Control Relat. Fields {\bf 4} (2014), no.~2, 161--186.
		
		\bibitem{Lebeau92} G. Lebeau, Contr\^ole de l'\'equation de Schr\"odinger, J. Math. Pures Appl. (9) {\bf 71} (1992), no.~3, 267--291.
		
		\bibitem{LZ82} B.~M. Levitan and V.~V. Zhikov, {\it Almost periodic functions and differential equations}, translated from the Russian by L. W. Longdon, Cambridge Univ. Press, Cambridge-New York, 1982.
		
		\bibitem{Lions88} J.-L. Lions, {\it Contr\^olabilit\'e{} exacte, perturbations et stabilisation de syst\`emes distribu\'es. Tome 1}, Recherches en Math\'ematiques Appliqu\'ees, 8, Masson, Paris, 1988.
		
		\bibitem{Maci09} F. Maci\`a, Semiclassical measures and the Schr\"odinger flow on Riemannian manifolds, Nonlinearity {\bf 22} (2009), no.~5, 1003--1020.
		
		\bibitem{Maci10}F. Maci\`a, High-frequency propagation for the Schr\"odinger equation on the torus, J. Funct. Anal. {\bf 258} (2010), no.~3, 933--955.
		
		\bibitem{Maci11} F. Maci\`a, The Schr\"odinger flow in a compact manifold: high-frequency dynamics and dispersion, in {\it Modern aspects of the theory of partial differential equations}, 275--289, Oper. Theory Adv. Appl. Adv. Partial Differ. Equ. (Basel), 216 , Birkh\"auser/Springer Basel AG, Basel.
		
		\bibitem{Mar06} J.~L. Marzuola, Eigenfunctions for partially rectangular billiards, Comm. Partial Differential Equations {\bf 31} (2006), no.~4-6, 775--790.
		
		\bibitem{Zwor12} M. Zworski, {\it Semiclassical analysis}, Graduate Studies in Mathematics, 138, Amer. Math. Soc., Providence, RI, 2012.
	\end{thebibliography}
\end{document}